\documentclass[a4paper,11pt,reqno]{amsart}

\usepackage{amsfonts}
\usepackage{amssymb}
\usepackage{amscd}
\usepackage{amsthm}
\usepackage{a4wide}
\usepackage{mathrsfs}

\usepackage[centertags]{amsmath}
\usepackage{eucal,dsfont,accents,bbm,amsfonts,amssymb,amsthm,amsopn}

\usepackage[usenames]{color}
\definecolor{citegreen}{rgb}{0,0.6,0}
\definecolor{refred}{rgb}{0.8,0,0}
\usepackage[colorlinks, citecolor=citegreen, linkcolor=refred]{hyperref}

\theoremstyle{plain}

\newtheorem{teo}{Theorem}[section]
\newtheorem{lemma}[teo]{Lemma}
\newtheorem{prop}[teo]{Proposition}
\newtheorem{cor}[teo]{Corollary}
\newtheorem{dfnz}[teo]{Definition}
\newtheorem{ackn}{Acknowledgments\!}

\theoremstyle{definition}

\newtheorem{rem}[teo]{Remark}

\theoremstyle{remark}

\numberwithin{equation}{section}

\def\NN{{{\mathbb N}}}

\def\RR{{\mathbb R}}

\def\RRR{{\mathrm R}}

\def\HHH{{\mathrm H}}

\def\Ric{{\mathrm {Ric}}}

\def\Rm{{\mathrm {Rm}}}

\def\eps{\varepsilon}

\def\div{\operatornamewithlimits{div}\nolimits}
\def\tr{\operatornamewithlimits{tr}\nolimits}
\def\dist{\mathrm{dist}}

\usepackage[applemac]{inputenc}
\usepackage{amsmath}
\usepackage{amssymb}
\usepackage{amsthm}
\usepackage{textcomp}
\usepackage{graphicx}
\usepackage{enumerate}
\usepackage{mathrsfs}
\usepackage{frcursive}
\usepackage{tikz}
\usepackage[cyr]{aeguill}
\usepackage{xspace}
\usepackage{hyperref}
\usepackage{appendix}

\def\Ric{\mathop{\rm Ric}\nolimits}
\def\Hess {\mathop{\rm Hess}\nolimits}
\def\Rm{\mathop{\rm Rm}\nolimits}
\def\tr{\mathop{\rm tr}\nolimits}

\def\vol{\mathop{\rm vol}\nolimits}

\def\vol{\mathop{\rm Vol}\nolimits}

\def\div{\mathop{\rm div}\nolimits}

\def\A{\mathop{\rm A}\nolimits}

\def\Li{\mathop{\rm \mathscr{L}}\nolimits}

\title[Uniqueness of asymptotically cylindrical gradient shrinking Ricci solitons]{Uniqueness of asymptotically cylindrical gradient shrinking Ricci solitons}

\author[Giovanni Catino]{Giovanni Catino}
\address[Giovanni Catino]{Dipartimento di Matematica, Politecnico di Milano, Piazza Leonardo da Vinci 32, 20133 Milano, Italy}
\email[]{giovanni.catino@polimi.it}

\author[Alix Deruelle]{Alix Deruelle}
\address[Alix Deruelle]{Mathematics Institute, University of Warwick, Gibbet Hill Rd, Coventry, West Midlands CV4 7AL}
\email{A.Deruelle@warwick.ac.uk}

\author[Lorenzo Mazzieri]{Lorenzo Mazzieri}
\address[Lorenzo Mazzieri]{Scuola Normale Superiore, Classe di Scienze, Piazza dei Cavalieri 7, 56126 Pisa, Italy}
\email{l.mazzieri@sns.it}

\date{\today}

\begin{document}

\begin{abstract} In this paper we prove that any asymptotically cylindrical gradient shrinking Ricci soliton is isometric to a cylinder.
\end{abstract}

\maketitle

\section{Introduction and statement of the result}
A gradient shrinking Ricci soliton is a smooth Riemannian manifold $(M^n, g)$ satisfying 
$$
\Ric \, + \, \nabla^{2}f \,=\, \lambda \, g \,,
$$
for some $\lambda>0 $ and some smooth function $f$ defined on $M^{n}$. In the following, we will adopt the normalization $\lambda = \frac{1}{2}$. Hence, throughout this paper the fundamental equation will be given by
\begin{equation}
\label{equ:0}
\Ric \, + \, \nabla^2 f \,\, = \,\, \frac{1}{2} \, g \, .
\end{equation}
Gradient shrinking Ricci solitons turn out to be Type I finite time singularities of the Ricci flow. Therefore their classification is important to understand the singularities of the Ricci flow in the large. Since the seminal work of Perelman on the classification of $3$-dimensional gradient shrinking Ricci solitons, there has been a vast amount of litterature on the subject (see \cite{Cao-Geo-Sol} for a survey).

In order to state our main result, we introduce the following definition of {\em asymptotically cylindrical} gradient shrinking Ricci soliton.
\begin{dfnz}
\label{a.c.}
A complete noncompact Riemannian manifold $(M^n, g)$ is said to be {\em asymptotically cylindrical} if for every sequence of marked points $(x_k)_{k\in \NN}$ which tends to infinity, 
the sequence of pointed Riemannian manifolds $(M^n, g, x_k)$ converges in the smooth Cheeger-Gromov sense to the cylinder $(\mathbb{R}\times \mathbb{S}^{n-1},dt^2+h)$, where $h$ is a metric of positive constant curvature.
\end{dfnz}

In the following we will consider cylindrical gradient shrinking Ricci solitons and we assume that $h$ is normalized in such a way that $\Ric_h=h/2$. We are now in the position to state our main result.

\begin{teo}\label{rev-asy-cyl-sgs}
\label{main}
Let $(M^n, g, f)$ be a complete noncompact gradient shrinking Ricci soliton, which is asymptotically cylindrical. Then, $(M^n, g, f)$ is isometric to the cylinder $(\mathbb{R}\times \mathbb{S}^{n-1},dt^2+h)$.
\end{teo}

In case of bounded positive curvature operator, we have the following corollary.

\begin{cor}\label{nonneg-cyl}
Let $(M^n, g, f)$ be a complete gradient shrinking Ricci soliton with bounded positive curvature operator. Then, it is either compact or $\limsup_{+\infty}\frac{\vol B(p,r)}{r}=+\infty$ for any $p\in M^n$.
\end{cor}

A first remark on Theorem \ref{rev-asy-cyl-sgs} is that we do not use any nonnegativity assumptions on the curvature tensor. Moreover, according to the classification of rotationally symmetric gradient shrinking Ricci solitons~\cite{kotschwar}, we are reduced to prove that any asymptotically cylindrical soliton is rotationally symmetric.

The core of the proof of Theorem \ref{rev-asy-cyl-sgs} is essentially based on the work of Brendle~\cite{Bre-Rot-3d}, where the author proves the uniqueness of the Bryant soliton in the class of gradient steady Ricci solitons with positive sectional curvature. Our proof is based on the ideas developed in~\cite{Bre-Rot-3d} with a substantially differences in the interpolation of almost-Killing vector fields (Section \ref{Vec-Interpol}) and the analysis of Lichnerowicz equation (Section \ref{Rig-Lic-Equ}). This allows us to drop the assumption on the nonnegativity of the curvature. Moreover, our approach is purely static, since it does not make use of the evolution equation of the geometric quantities under the Ricci flow. 

Now, we present the main steps of the proof of Theorem \ref{rev-asy-cyl-sgs}.

\begin{itemize}
\item To prove the rotational symmetry, it is sufficient to build $n(n-1)/2$ independent Killing vector fields orthogonal to $X:=\nabla f$. Generally speaking, a Killing field $U$ satisfies $\Delta U+\Ric(U)=0$. In the case of a gradient shrinking Ricci soliton, this gives $\Delta U-\nabla_UX+U/2=0$. For technical reasons, it is easier to estimate the operator $\Phi:U\mapsto\Delta U-\nabla_X U+U/2$. Using the assumption on the asymptotic behavior, Sections \ref{1-cov-der} and \ref{Alm-Kill-Fiel} are devoted to build $n(n-1)/2$ vector fields $\{U_i\}_i$ which are almost Killing (Proposition \ref{prop-ex-alm-kill}).  In particular, it is shown that the vector fields $U_i$ are bounded and the vector fields $\Phi(U_i)$ decay sufficiently fast. 

\item In Section \ref{Vec-Interpol}, we prove Theorem \ref{theo-ex-app-kill}  that establishes the surjectivity of $\Phi$ in the following sense: we prove the existence of vector fields $V_i$ decaying sufficiently fast  such that $\Phi(V_i)= \Phi(U_i)$ for any $i=1,\ldots n-1$. In particular, this ensures that the vector fields $U_i-V_i$ are not trivial. For that purpose, we use the potential function as a barrier to establish a maximum principle at infinity (Proposition \ref{max-ppe-vec}).

\item In Section \ref{Rig-Lic-Equ} we study the rigidity of Lichnerowicz equation. In fact, the vector fields $W_i:=V_i-U_i$ built previously lie in $\ker \Phi$. Now, $\ker\Phi$ is not reduced to Killing fields since $X\in\ker\Phi$. Theorem \ref{triv-ker} shows that there exist real numbers $\lambda_i$ such that the vector fields $W_i-\lambda_iX$ are Killing. The proof consists in noting that the symmetric $2$-tensors $\Li_{W_i}g=:h_i$ satisfy Lichnerowicz equation $\Li_X h_i -h_i=\Delta_L(h_i)$ where $\Delta_L$ is the Lichnerowicz Laplacian. Theorem \ref{triv-ker} shows that the only solution decaying polynomially is, up to a homothety, $\Li_X g $. As for vector fields, we need to establish a priori estimates (Proposition \ref{c0-est-lic}) with the help of the barrier function $v:=f-n/2$ which is a positive eigenfunction of $\Delta_f$, i.e. $\Delta_fv=-v$.

\item In Section \ref{Conc} we conclude the proof of Theorem \ref{rev-asy-cyl-sgs} and Corollary \ref{nonneg-cyl}.
\end{itemize}
We end the introduction with some remarks on related works. Recently, Kotschwar and Wang \cite{Kot-Wan} proved that two gradient shrinking Ricci solitons whose asymptotic cones are isometric are actually isometric. Their method is completely different and more involved, since they are dealing with the most general situation of an asymptotically conical metric. Combining our proof with rescaling arguments, one can reprove the result of Kotschwar and Wang in the particular case where the asymptotic cones are the most symmetric ones. More precisely, we obtain the following theorem.

\begin{teo}
Any smooth gradient shrinking Ricci soliton which is asymptotic to the cone $$(C(\mathbb{S}^{n-1}),dr^2+(cr)^2g_{\mathbb{S}^{n-1}}),$$ where $g_{\mathbb{S}^{n-1}}$ is the metric of constant sectional curvature $1$ and $c>0$, is rotationally symmetric. Hence $c=1$ and the metric is isometric to the Euclidean space.
\end{teo}

We would like to mention that our method applies also to the class of expanding gradient Ricci solitons (EGS for short). We recall that an EGS is a Riemannian manifold $(M^n, g)$ satisfying 
$$
\Ric \, - \, \nabla^{2}f \,=\,- \frac{1}{2}\, g \,,
$$
for some smooth function $f$ defined on $M^{n}$. Such solutions naturally arise as blow-up of noncompact noncollapsed Type III singularities with nonnegative curvature operator according to the work of Schulze and Simon \cite{Sch-Sim}. As a consequence of their work, the study of the asymptotic geometry of noncompact noncollapsed Riemannian manifolds with nonnegative curvature operator reduces to the classification of (the asymptotic cones of) nonnegatively curved EGS.
Now, Bryant, in unpublished notes, has also built  a one-parameter family of rotational symmetric EGS on $\mathbb{R}^{n}$ for $n\geq 3$ asymptotic to $(C(\mathbb{S}^{n-1}),dr^2+(cr)^2g_{\mathbb{S}^{n-1}}),$ with $c>0$ (see e.g.~\cite[Section 5, Chap.1]{Cho-Lu-Ni-I}). Therefore, it is natural to ask for the analogue of the Kotschwar-Wang theorem in this setting. With the above notation, this amounts to prove that a gradient expanding Ricci soliton which is asymptotic to the cone $(C(\mathbb{S}^{n-1}),dr^2+(cr)^2g_{\mathbb{S}^{n-1}})$ is rotationally symmetric. Recently, Chodosh \cite{Cho-EGS} answered positively in the case the metric has nonnegative curvature operator. Apparently this assumption is only needed in the proof of~\cite[Proposition 5.1]{Cho-EGS}. This proposition can be replaced following our arguments in Section~\ref{Rig-Lic-Equ}, where the key ingredient is the existence of an eigenfunction of $\Delta_{f}$. In the EGS case, such a function is given by $f+n/2$. Eventhought, an estimate à la Cao-Zhou \cite{cao-zho} does not hold for a general EGS, it is known that for asymptotical conical EGS, the potential function still behaves like $r^2/4$, where $r$ is the distance to a fixed point, since the curvature is decaying quadratically (see \cite{Der-Che} for more details). Hence, $f+n/2$ can still play the role of a barrier function in the proof of all the a priori estimates. In particular, we can get rid of the positive curvature assumption in~\cite{Cho-EGS}, obtaining the following theorem.

\begin{teo}
Any gradient expanding Ricci soliton which is asymptotic to the cone $$(C(\mathbb{S}^{n-1}),dr^2+(cr)^2g_{\mathbb{S}^{n-1}}),$$ where $g_{\mathbb{S}^{n-1}}$ is the metric of constant sectional curvature $1$ and $c>0$ is rotationally symmetric.
\end{teo}


\

\

\begin{ackn} 
The second author is supported by the EPSRC on a Programme Grant entitled ``Singularities of Geometric Partial Differential Equations'' (reference number EP/K00865X/1).
The third author is partially supported by the Italian project FIRB 2012 ``Geometria Differenziale e Teoria Geometrica delle Funzioni'' as well as by the SNS project ``Geometric flows and related topics''. We are grateful to Otis Chodosh for a careful reading of the first version of the manuscript and for pointing out some inaccuracies.
\end{ackn}

\

\section{Asymptotic geometry}
\label{1-cov-der}

We start by recalling some basic and well known curvature identities that hold on a gradient shrinking Ricci soliton. A proof of these identities can be found for example in~\cite{Emi-LaN-Man}.
\begin{lemma}
\label{id-SGS}
Let $(M^n, g, f)$ be a gradient shrinking Ricci soliton. Then, setting $X:=\nabla f$, the following identities hold true.
\begin{eqnarray}
\Delta f \, + \, \RRR & = & \frac{n}{2} \,, \label{equ:1} \\
\nabla \RRR & = & 2\Ric(X , \, \cdot \,) \, , \label{equ:2} \\
\arrowvert X \arrowvert^2\, + \, \RRR & = & f, \label{equ:3}\\
\div\Rm \, (Y,Z,W) & = & \Rm \, (Y,Z, W,X),\label{equ:4}
\end{eqnarray}
for every vector fields $Y$, $Z$, $W$. Moreover, setting  $\Delta_f:=\Delta-\nabla_X$, we have that
\begin{eqnarray}
\Rm & = & \Delta_f \Rm \, + \,\Rm\ast\Rm,\label{equ:5}\\
\Ric & = & \Delta_f\Ric \, + \, 2\Rm\ast\Ric,\label{equ:6}\\
\RRR & = & \Delta_f \RRR \, + \, 2\arrowvert\Ric\arrowvert^2\label{equ:7},
\end{eqnarray}
where, if $A$ and $B$ are two tensors, $A\ast B$ denotes some linear combination of contractions of the tensorial product of $A$ and $B$.
\end{lemma}
\begin{rem}
We observe that the general form of identity~\eqref{equ:3} is $\arrowvert X \arrowvert^2\, + \, \RRR \, = \, f \, + \, c$, where $c\in \RR$ is a real constant. In the rest of this paper we will systematically make the normalization assumption $c=0$.
\end{rem}

%
%
%
%
%

We recall the following growth estimate on the potential function of a noncompact gradient shrinking soliton due to Cao-Zhou \cite{cao-zho}.

\begin{lemma} Let $(M^n, g, f)$ be a complete noncompact gradient shrinking Ricci soliton. Then, the potential function $f$ satisfies the estimates
$$
\frac{1}{4} (r(x)-c_1)^2\leq f(x)\leq \frac{1}{4} (r(x)+c_2)^2 \,,
$$
where $r(x)=d(x_0, x)$ is the distance function from some fixed point $x_0\in M$, $c_1$ and $c_2$ are positive constants depending only on $n$ and the geometry of $g$ on the unit ball $B(x_0,1)$.
\end{lemma}

From now on, we assume that our complete noncompact gradient shrinking Ricci soliton $(M^n,g,f)$ is {\em asymptotically cylindrical} in the sense of Definition~\ref{a.c.}. In order to give a more careful estimate  of how the soliton metric converges to the cylindrical one, it is convenient to introduce the following tensor
\begin{eqnarray}
\label{T}
T \,\, := \,\, \Ric \,\, - \, \, \frac{\RRR}{n-1} \,  \bigg(\,  g \, - \,  \frac{df \otimes df}{|df|^2}  \, \bigg) \, .
\end{eqnarray}
We observe that the tensor $T$ is well defined whenever $\nabla f \neq 0$. In particular, from the results in~\cite{cao-zho}, we have that $T$ is well defined outside a compact set. 

\begin{lemma} 
\label{Ttozero}
Let $(M^n, g, f)$ be a complete noncompact asymptotically cylindrical gradient shrinking Ricci soliton and let $T$ be the tensor defined in~\eqref{T} outside a compact set. Then, we have that $|T| = o \, (1)$, as $f \to +\infty$. This means that for every sequence of points $(x_k)_{k\in \mathbb{N}}$ such that $f(x_k) \to +\infty$, as $k\to +\infty$, one has that $|T| (x_k) \to 0$.
\end{lemma}
\begin{proof}
We start by computing the quantity $|T|^2$.
\begin{eqnarray*}
|T|^2 & = & |\Ric|^2 \, - \frac{\RRR^2}{n-1}  \, + \, \frac{ \RRR}{n-1} \, \frac{2 \, \Ric (\nabla f, \nabla f)}{|\nabla f|^2} \\
& = & |\Ric|^2 \, - \frac{\RRR^2}{n-1} \, + \,  \frac{ \RRR}{n-1} \, \frac{ \langle \nabla \RRR \,, \nabla f \rangle}{|\nabla f|^2} \, ,
\end{eqnarray*}
where we used the identity~\eqref{equ:3} in the last equality. By the fact that the soliton is asymptotically cylindrical, it is immediate to deduce that  $|\Ric|^2 \, - {\RRR^2}/({n-1}) = o \,(1)$, for $f \to + \infty$. For the same reason, we have that $|\nabla \RRR| = o \, (1)$, as $f\to + \infty$, whereas, by the results in~\cite{cao-zho}, one has that $|\nabla f|^2 = O\,(f)$, as $f\to + \infty$. From these facts, we infer that also the third term in the right hand side tends to zero, as $f\to + \infty$.
\end{proof}

So far, we have shown that $|T| = o\, (1)$ for $f\to +\infty$. In order to improve this estimate we state the following lemma, in which we prove some basic but useful properties of the tensor $T$.

\begin{lemma} 
\label{T_decay}
Let $(M^n, g, f)$ be a complete noncompact asymptotically cylindrical gradient shrinking Ricci soliton and let $T$ be the tensor defined in~\eqref{T} outside a compact set. Then, setting $\mathbf{n} := X/|X| = \nabla f/ |\nabla f|$, we have 
\begin{eqnarray}
\tr T=0 \, ,\quad\quad |\,T(\mathbf{n}, \cdot \,)\,| =\textit{o}\, (f^{-1/2}) \quad \hbox{and}\quad T(\mathbf{n},\mathbf{n})=\textit{o}\, (f^{-1}) \, ,
\end{eqnarray}
as $f\rightarrow + \infty$.
\end{lemma}

\begin{proof} The fact that the tensor $T$ is traceless follows immediately from its definition. Using the identities~\eqref{equ:2} and~\eqref{equ:3} in Lemma~\ref{id-SGS}, we get
\begin{eqnarray*}
|\,T(\mathbf{n}, \cdot \, )\,| \, = \, |\Ric(\mathbf{n} , \cdot \, ) \, | \, = \, \frac{|\nabla \RRR|}{2\arrowvert X\arrowvert} \, = \, \textit{o}\, (f^{-1/2}) \, ,
\end{eqnarray*}
where in the last equality we also used the fact that the soliton is {\em asymptotically cylindrical} and thus $|\nabla \RRR| \rightarrow 0$, as $f \to +\infty$. From the asymptotic behavior of the soliton it is also possible to deduce that $\Delta \RRR \to 0$, as $f \to + \infty$. Moreover, by the Definition~\ref{a.c.}, one has $|\Ric - \frac{1}{2}g \,  | \rightarrow 0 $ and thus $2|\Ric|^2 - \RRR \to 0$, as $f \to + \infty$. Combining this with the identities above and with equation~\eqref{equ:7}, we deduce
\begin{eqnarray*}
T(\mathbf{n},\mathbf{n})\,\,= \,\,\Ric(\mathbf{n},\mathbf{n}) \,\, = \,\, \frac{\langle\nabla \RRR,X\rangle}{\arrowvert X\arrowvert^2} \,\, = \,\, \frac{(\Delta \RRR+2\arrowvert \Ric \arrowvert^2-\RRR)}{\arrowvert X\arrowvert^2} \,\, = \,\, \textit{o}\, (f^{-1}) \, .
\end{eqnarray*}
This completes the proof of the lemma.
\end{proof}

In the next proposition, we derive a partial differential inequality for the quantity $|T|^2$. This will then be used to improve the estimates of the decay of $|T|$ at infinity.

\begin{prop}
\label{PDE}
Let $(M^n, g, f)$ be a complete noncompact asymptotically cylindrical gradient shrinking Ricci soliton and let $T$ be the tensor defined in~\eqref{T} outside a compact set. Then, there exists a positive constant $c(n)$, only depending on the dimension $n$, such that
\begin{eqnarray*}
\Delta_f\arrowvert T\arrowvert^2 \,\, \geq \,\, c(n) \, \arrowvert T\arrowvert^2 \, + \, \textit{O}\, (f^{-1}) \,,
\end{eqnarray*}
as $f \to \infty$. Moreover, we can choose $c(n)=2/(n-2) - \eta$ for any positive $\eta$ sufficiently small.
\end{prop}

\begin{proof}
Using the same notations as in Lemma~\ref{T_decay} and taking advantage of the equations~\eqref{equ:6} and~\eqref{equ:7}, we compute the difference $\Delta_f T-T $, namely
\begin{eqnarray*}
\Delta_f T-T &=& \Delta_f \Ric \, - \, \Ric \, - \, \frac{1}{n-1} \, \big[ \, \Delta_f (\RRR(g-\mathbf{n}\otimes\mathbf{n})) \, - \, \RRR(g-\mathbf{n}\otimes\mathbf{n}) \, \big]\\
&=& -\frac{1}{n-1} \, \big[ \, (\Delta_f \RRR-\RRR)(g-\mathbf{n}\otimes\mathbf{n}) \, + \,  
2\, \nabla_{\nabla \RRR} (\mathbf{n}\otimes\mathbf{n}) \, - \, \RRR \, \Delta_f(\mathbf{n}\otimes\mathbf{n}) \, \big] \\
& & -2\Rm\ast\Ric
\\
&=&\frac{1}{n-1} \, \big[ \, 2\arrowvert\Ric\arrowvert^2(g-\mathbf{n}\otimes\mathbf{n})
\, - \, 
2 \, \nabla_{\nabla \RRR}( \mathbf{n}\otimes\mathbf{n}) \, + \, \RRR \, \Delta_f(\mathbf{n}\otimes\mathbf{n}) \, \big] \\
& & -2\Rm\ast\Ric \, .
\end{eqnarray*}
Now, we recall that $(\Rm\ast\Ric)_{ij}:=\RRR_{iklj}\RRR_{kl}$. Therefore, we get
\begin{eqnarray*}
\langle\Rm\ast\Ric \, ,T\rangle &=& \langle \Rm\ast T \, ,T\rangle \, + \, \frac{R}{n-1} \, \langle\Rm\ast(g-\mathbf{n}\otimes\mathbf{n}) \, ,T\rangle\\
&=&\langle\Rm\ast T \, ,T\rangle \, + \, \frac{\RRR}{n-1} \, \langle \Ric-\Rm(\mathbf{n},\cdot,\cdot,\mathbf{n}) \, ,T\rangle\\
&=&\langle\Rm\ast T \, ,T\rangle \, + \, \frac{\RRR}{n-1} \, \arrowvert T\arrowvert^2 \, - \, \frac{\RRR^2}{(n-1)^2} \, T(\mathbf{n},\mathbf{n}) \, - \, \frac{\RRR}{n-1} \, \langle \Rm(\mathbf{n},\cdot,\cdot,\mathbf{n}) \, , T \rangle \,,
\end{eqnarray*}
where in the last equality we used the fact that the tensor $T$ is traceless.
Since $(M^n,g, \nabla f)$ is {\em asymptotically cylindrical}, we use the Lemma~\ref{Ttozero} and the fact that the Weyl part of the Riemmann tensor tends to zero in order to deduce that
\begin{eqnarray*}
\Rm \,\, = \,\, \frac{\RRR}{(n-1)(n-2)}(g-\mathbf{n}\otimes\mathbf{n})\odot(g-\mathbf{n}\otimes\mathbf{n})\, +\textit{o}(1) \, ,
\end{eqnarray*}
where $\odot$ represents the Kulkarni Nomizu product. As a consequence, we get
\begin{eqnarray*}
\langle\Rm\ast T,T\rangle & = & \frac{\RRR}{(n-1)(n-2)} \, \left[ \, 2 \, |\,T(\mathbf{n}, \cdot)|^2 - \, |T|^2
\, \right] \, +\, \textit{o} \, (\arrowvert T\arrowvert^2)\\
&=&- \, \frac{\RRR}{(n-1)(n-2)} \, \arrowvert T\arrowvert^2 \, +\, \textit{o} \, (\arrowvert T\arrowvert^2) \, + \, \textit{o} \, (f^{-1}) \, ,
\end{eqnarray*}
where in  the last equality we have used the estimates in Lemma~\ref{T_decay}. Moreover, using equation (\ref{equ:4}) shows that 
$$
 \langle \Rm(\mathbf{n},\cdot,\cdot,\mathbf{n}) \, , T \rangle \,\, = \,\, \frac{1}{\arrowvert X\arrowvert} \, \div\Rm\ast T   \,\, = \,\,  o \, (f^{-1/2})\arrowvert T\arrowvert \, .
$$
Taking the sum, we obtain
\begin{eqnarray}
\label{expansion}
\Delta_f\arrowvert T\arrowvert^2-2\arrowvert T\arrowvert^2 & = & 2\arrowvert\nabla T\arrowvert^2  \, - \,  \frac{4(n-3) \, \RRR}{(n-1)(n-2)} \, \arrowvert T\arrowvert^2 \, +\, \textit{o} \, (\arrowvert T\arrowvert^2)+o \, (f^{-1/2})\arrowvert T\arrowvert \,   \nonumber
\\
& & + \, \frac{2}{(n-1)} \, \Big\langle -\, 2 \, \nabla_{\nabla \RRR}( \mathbf{n}\otimes\mathbf{n}) \, + \, \RRR \, (\Delta_f(\mathbf{n}\otimes\mathbf{n}) \, , \, T \, \Big\rangle \, + \, \textit{o} \,(f^{-1}) \, . \hspace{1 cm} 
\end{eqnarray}
In order to proceed, we are going to analyze the asymptotic behavior of the second raw. We claim that
$$
\langle 
-\,2 \, \nabla_{\nabla \RRR}( \mathbf{n}\otimes\mathbf{n}) \, + \, \RRR \,  \Delta_f(\mathbf{n}\otimes\mathbf{n}) \, ,T\rangle \,\, = \,\, \textit{O} \, (f^{-1}) \, ,
$$
as $f\to \infty$. We start with the estimate of the term $\langle
2 \, \nabla_{\nabla \RRR}( \mathbf{n}\otimes\mathbf{n})  \, ,T\rangle$. First, we notice that
\begin{eqnarray*}
\langle \nabla_{\nabla \RRR}( \mathbf{n}\otimes\mathbf{n})  \, ,T\rangle & = & \langle \nabla_{\nabla\RRR} \, \mathbf{n} \, , T(\mathbf{n}, \cdot)  \rangle \,\, \leq \,\, | \nabla_{\nabla\RRR} \, \mathbf{n} | \, |T(\mathbf{n}, \cdot) | \,\, \leq \,\, |\nabla \mathbf{n}| \, |\nabla \RRR| \, |T(\mathbf{n}, \cdot) |\,.
\end{eqnarray*}
On the other hand, we have that 
\begin{eqnarray*}
|\nabla \mathbf{n}| & = & \left| \, \frac{\nabla\nabla f}{|\nabla f|}   \, - \, \frac{\nabla \nabla f (\mathbf{n} \, , \cdot \,) \otimes \mathbf{n}}{|\nabla f|}    \,\right| \,\,\, \leq \,\,\,  2 \, \frac{|\nabla \nabla f|}{|\nabla f|} \,\,\, = \,\,\, O \, (f^{-1/2}) \, .
\end{eqnarray*}
Combining this with Lemma~\ref{T_decay}, we obtain that $\langle \nabla_{\nabla \RRR}( \mathbf{n}\otimes\mathbf{n})  \, ,T\rangle \, = \, o\, (f^{-1})$.
We pass now to estimate the term $\langle \RRR \,  \Delta_f(\mathbf{n}\otimes\mathbf{n}) \, ,T\rangle \, = \, \langle \RRR \,  \Delta(\mathbf{n}\otimes\mathbf{n}) \, ,T\rangle \, - \, \langle \RRR \,  \nabla_X (\mathbf{n}\otimes\mathbf{n}) \, ,T\rangle $. We start with the second term of the right hand side.
\begin{eqnarray*}
\langle\nabla_X(\mathbf{n}\otimes\mathbf{n})\, ,T\rangle & = & 2\, \langle\nabla_X\mathbf{n}\, ,T(\mathbf{n} , \cdot)\rangle   \\
& \leq & 2 \, \left| \,\frac{ \langle \nabla\nabla f(X, \cdot\,) \, ,     T(\mathbf{n}, \cdot)      \rangle}{|\nabla f|}   \, - \, \frac{\nabla \nabla f (\mathbf{n} \, , X) \,T(\mathbf{n}, \mathbf{n})}{|\nabla f|}    \,\right|      \\
& \leq & 2 \, \left| \,\frac{ \langle \Ric(X, \cdot\,) \, ,     T(\mathbf{n}, \cdot)      \rangle}{|\nabla f|}  \, \right| \, + \, |T(\mathbf{n}, \mathbf{n})| \, + \, \left|  \frac{\nabla \nabla f (\mathbf{n} \, , X) \,T(\mathbf{n}, \mathbf{n})}{|\nabla f|}    \,\right|      \\
&  \leq & \, \frac{|\nabla \RRR |}{|\nabla f|} \, |T(\mathbf{n}, \cdot) |  \, + \, 2 \, |T(\mathbf{n}, \mathbf{n})| \,\, = \,\,\textit{o}\, (f^{-1})\, ,
\end{eqnarray*}
where we used the identity~\eqref{equ:2} and the estimates in Lemma~\ref{T_decay}. To estimate the term $ \langle \RRR \,  \Delta(\mathbf{n}\otimes\mathbf{n}) \, ,T\rangle $, we recall that
\begin{eqnarray*}
\Delta(\mathbf{n}\otimes\mathbf{n}) \,\, = \,\, (\Delta\mathbf{n})\otimes\mathbf{n} \, + \, \mathbf{n}\otimes (\Delta\mathbf{n}) \, + \, 2 \, \nabla\mathbf{n}\ast\nabla\mathbf{n}
\end{eqnarray*}
and we immediately notice that $\langle \nabla\mathbf{n}\ast\nabla\mathbf{n}\, ,T \rangle \, = \, \nabla_k \mathbf{n}_i\, \nabla_k\mathbf{n}_j \, T_{ij} \, = \, \textit{O}\, (f^{-1})$. Moreover, a direct computation shows that
\begin{eqnarray*}
\Delta\mathbf{n} & = & \arrowvert X \arrowvert \, \Delta\left(\frac{1}{\arrowvert X\arrowvert}\right) \, \mathbf{n} \, + \, 2 \, \nabla_k\left(\frac{1}{\arrowvert X\arrowvert}\right)\cdot\nabla_kX \, + \, \frac{1}{\arrowvert X\arrowvert}\Delta X \, .
\end{eqnarray*}
Using the identities in Lemma~\ref{T_decay}, it is easy to obtain $\Delta X \, = \, \Delta \nabla f\, = \, \nabla \Delta f \, + \, \Ric(\nabla f , \cdot 	\, ) 	\, = \, -\nabla \RRR \, + \, \Ric(X\, , \cdot \, ) \, = \, - \frac{1}{2}\nabla \RRR$, and thus  
\begin{eqnarray*}
\frac{|\Delta X|}{\arrowvert X\arrowvert} \,\, = \,\, \textit{o} \, (f^{-1/2}) \, .
\end{eqnarray*}
Similarly, $|\nabla(1/\arrowvert X\arrowvert)|  = \textit{O}\, (f^{-1})$ and $\arrowvert X\arrowvert \, \Delta(1/\arrowvert X\arrowvert) \, = \, \textit{O}\, (f^{-1})$. Recasting all these estimates, it is straightforward to check that the claim is proven. Combining the claim with the estimate~\eqref{expansion}, we obtain
\begin{eqnarray*}
\Delta_f\arrowvert T\arrowvert^2 & = & 2\,\arrowvert\nabla T\arrowvert^2  \, + \,  \frac{2}{(n-2)} \, \arrowvert T\arrowvert^2 \, +\, \textit{o} \, (\arrowvert T\arrowvert^2)  \, + \, \textit{O} \,(f^{-1}) \, .
\end{eqnarray*}
The statement of the proposition follows at once.
\end{proof}
In order to analyze the partial differential inequality obtained in Proposition~\ref{PDE}, we prove the following algebraic lemma about the ``evolution'' of the curvature tensor of a gradient Ricci soliton. A parabolic proof of this result can be found for example in~\cite{Ben}, where the time derivative plays the role of the covariant derive along $\nabla f$.

\begin{lemma}
\label{lemm-sol-id}
Let $(M^n, g, f)$ be a gradient Ricci soliton. Then, for every vector fields $U$, $V$, $W$ and $Y$,
\begin{eqnarray*}
\nabla_{\nabla f}\Rm(U,V,W,Y) & = & -\nabla_U (\div\Rm)  (W,Y,V) \, + \, \nabla_{V} (\div\Rm) (W,Y,U)\\
& & + \Rm(V,\nabla^2 f(U , \cdot \,),W,Y) \, - \, \Rm(U,\nabla^2 f(V, \cdot \,),W,Y).
\end{eqnarray*}
\end{lemma}
\begin{proof}
By the second Bianchi identity, we have
\begin{eqnarray*}
\nabla_{\nabla f}\Rm(U,V,W,Y) \,\, = \,\, -\nabla_U\Rm(V,\nabla f,W,Y) \, - \, \nabla_{V}\Rm(\nabla f,U,W,Y) \, .
\end{eqnarray*}
On the other hand, one has
\begin{eqnarray*}
\nabla_U\Rm(V,\nabla f,W,Y) &  = & U \, \big( \Rm(V,\nabla f,W,Y)  \, \big) \, - \, \Rm(\nabla_UV,\nabla f,W,Y)\\
&&-\, \Rm(V,\nabla^2 f(U, \cdot \,),W,Y)\, - \, \Rm(V,\nabla f,\nabla_UW,Y)\\
&&- \, \Rm(V,\nabla f,W,\nabla_UY)\\
&=&\nabla_U(\div\Rm) (W,Y,V) \, - \, \Rm(V,\nabla^2 f(U , \cdot \, ),W,Y) \, ,
\end{eqnarray*}
where, in the last equality we used the identity~\eqref{equ:4} in Lemma~\ref{id-SGS}. This concludes the proof of the lemma. 
\end{proof}
As a corollary of this general lemma, we obtain estimates for the covariant derivatives of the Ricci tensor along $\nabla f$.
\begin{cor}
\label{cor1}
Let $(M^n, g, f)$ be a complete noncompact asymptotically cylindrical gradient shrinking Ricci soliton. Then, we have that $|\nabla_{\nabla f} \Ric|$ and $|\nabla_{\nabla f} \nabla_{\nabla f} \Ric|$ are bounded.
\end{cor}
\begin{proof}
First of all we notice that, in local coordinates, the statement of the previous lemma reads
\begin{eqnarray*}
\nabla_p f \, \nabla_p \RRR_{ijkl} & = & - \, \nabla_i (\div \Rm)_{klj} \, + \, \nabla_j (\div \Rm)_{kli} \, + \, \nabla_i\nabla_p  f \, \RRR_{jpkl} \, - \, \nabla_j\nabla_p f \, \RRR_{ipkl} \\
& = &  - \, \nabla_i  (\nabla_k \RRR_{lj}  -  \nabla_{l} \RRR_{kj} ) \, + \, \nabla_j  (\nabla_k \RRR_{li}  -  \nabla_{l} \RRR_{ki} ) \\
& & -  \, \RRR_{ip} \, \RRR_{jpkl} \, + \, \RRR_{jp} \, \RRR_{ipkl} \, -  \, \RRR_{ijkl} \, ,
\end{eqnarray*}
where, in the last equality, we used the contracted second Bianchi identity and the soliton equation~\eqref{equ:0}. By the fact that the soliton is asymptotically cylindrical, we obtain that $|\nabla_{\nabla f} \Rm|$ is bounded. In particular, we have that also $|\nabla_{\nabla f} \Ric|$ is bounded. Taking the trace of the identity above, we get
\begin{eqnarray}
\label{equaz2}
\nabla_p f \, \nabla_p \RRR_{jk} & = &  - \, \nabla_i  \nabla_k \RRR_{ij} \, + \,   \Delta \RRR_{kj}  \, + \, \frac{1}{2} \, \nabla_j \nabla_k \RRR  \, -  \, \RRR_{ip} \, \RRR_{jpki} \, + \, \RRR_{jp} \, \RRR_{pk} \, -  \, \RRR_{jk} \,.
\end{eqnarray}
We are now in the position to estimate $|\nabla_{\nabla f} \nabla_{\nabla f} \Ric|$. In fact, taking the derivative of the previous expression, we obtain
\begin{eqnarray*}
\nabla_q f \, \nabla_p f \, \big( \nabla_q \nabla_p \RRR_{jk} \big) 
& = & \nabla_q f \, \nabla_q \big(  \nabla_p f \, \nabla_p \RRR_{jk}  \big) \, - \, \nabla_q f \, (\nabla_q \nabla_p f) \, (\nabla_p \RRR_{jk}) \\
& = & \nabla_q f \,\, \nabla_q \big(  - \, \nabla_i  \nabla_k \RRR_{ij} \, + \,   \Delta \RRR_{kj}  \, + \, \frac{1}{2} \, \nabla_j \nabla_k \RRR \big)\\
& &  + \, \nabla_q f \,\, \nabla_q \big(   -  \, \RRR_{ip} \, \RRR_{jpki} \, + \, \RRR_{jp} \, \RRR_{pk} \, -  \, \RRR_{jk}  \, \big) \\
& &   + \, \frac{1}{2} \, \nabla_p \RRR  \, (\nabla_p \RRR_{jk} )  \, - \, \frac{1}{2} \,  \nabla_p f  \, (\nabla_p \RRR_{jk} )    \,.
\end{eqnarray*}
From the estimates obtained in the first part of the proof, we infer that the second  and the third raw of the right hand side are bounded. To estimate the first raw, we first notice that, by the contracted second Bianchi identity, we have that
$$
\frac{1}{2} \, \nabla_j \nabla_k \RRR \,  - \, \nabla_i  \nabla_k \RRR_{ij}  \,\, = \,\, \RRR_{ikjl} \RRR_{il}
\, - \,  \RRR_{ik}\RRR_{ij} \,.
$$
Hence, reasoning as before, it is easy to deduce that the term $ \nabla_q f \, \nabla_q\big( \,\frac{1}{2} \, \nabla_j \nabla_k \RRR \,  - \, \nabla_i  \nabla_k \RRR_{ij} \,  \big)$ is bounded. To complete the proof, we thus need to estimate the term $\nabla_q f \, (\nabla_q \Delta \RRR_{kj})$. By the usual formulae for the exchange of the derivatives, we get
\begin{eqnarray}
\label{equaz} 
\nabla_q \Delta \RRR_{kj} & = & \Delta \nabla_q \RRR_{kj} \, - \, (\div \Rm)_{qkl} \RRR_{lj} \, - \, (\div \Rm)_{qjl} \RRR_{lk} \\
& & + \, \RRR_{ql} (\nabla_{l}\RRR_{kj}) \, + \, \RRR_{qpkl} (\nabla_{p}\RRR_{lj}) \, + \, \RRR_{qpjl} (\nabla_{p}\RRR_{kl}) \, . 
\end{eqnarray}
Using the identites~\eqref{equ:2} and~\eqref{equ:4}, it is immediate to check that the last raw contracted with $\nabla_q f$ gives rise to bounded terms. On the other hand we have that
\begin{eqnarray*}
\nabla_q f \, (\div \Rm)_{qkl} & = & \nabla_q f \, \big(  \nabla_q \RRR_{kl} - \nabla_k \RRR_{ql}   \big) \\
& = &   \nabla_q f \,   (\nabla_q \RRR_{kl}) \, - \, \nabla_k (\nabla_q f \, \RRR_{ql}) \, + \, \nabla_k \nabla_q f \, \RRR_{ql} \\
& = & \nabla_q f \,   (\nabla_q \RRR_{kl}) \, - \, \frac{1}{2}\, \nabla_k \nabla_l \RRR - \RRR_{kq} \RRR_{ql} + \frac{1}{2}\RRR_{kl} \,,
\end{eqnarray*}
and thus, by the previous discussion, it is evident that the second and the third terms in the first raw of~\eqref{equaz} contracted with $\nabla_{q} f$ are bounded. Finally, we have
\begin{eqnarray*}
(\Delta \nabla_q \RRR_{kj}) \, \nabla_{q} f &=& \Delta \big( \nabla_q \RRR_{kj} \, \nabla_{q} f  \big) - \nabla_q \RRR_{kj} ( \Delta \nabla_{q} f ) - 2 \nabla_{p} \nabla_{q} R_{kj} \, \nabla_{p} \nabla_{q} f \\
&=&  \Delta \big( \nabla_q \RRR_{kj} \, \nabla_{q} f  \big) - \nabla_q \RRR_{kj} ( \nabla_{q} \Delta f ) - \frac{1}{2}\nabla_q \RRR_{kj} \,\nabla_{q} \RRR + 2 \nabla_{p} \nabla_{q} R_{kj} \, \RRR_{pq} - \Delta \RRR_{kj} \\
&=&  \Delta \big( \nabla_q \RRR_{kj} \, \nabla_{q} f  \big) + \frac{1}{2} \nabla_q \RRR_{kj} \nabla_{q} \RRR + 2 \nabla_{p} \nabla_{q} R_{kj} \, \RRR_{pq} - \Delta \RRR_{kj} \,.
\end{eqnarray*}
Again, from equation~\eqref{equaz2} and the previous observations, we have that all the terms of the right hand side are bounded and this completes the proof of the corollary.
\end{proof}

We are now in the position to prove the following quantitative decay estimates of $|T|$ and $|\nabla^{k} T|$ as $f \to \infty$.
\begin{prop}
\label{Tdecay}
Let $(M^n, g, f)$ be a complete noncompact asymptotically cylindrical gradient shrinking Ricci soliton and let $T$ be the tensor defined in~\eqref{T} outside a compact set. Then, there exists a positive constant $a(n)$ such that
$$
|T|^{2} \,=\, O\,(f^{-a(n)} ) \quad\quad \hbox{and} \quad\quad |\nabla^{k} T|^{2} \,=\, O\, ( f^{-a(n)+\eps}) \,,
$$
for every $k>0$ and every $\eps>0$. Moreover, we can choose $a(n):=\min\{1, 2/(n-2) - \eta \}$, for any positive $\eta$ sufficiently small.
\end{prop}
\begin{proof}
We have already observed that, for $f\geq f_{0}$ large enough, $|\nabla f|>0$ and thus $\{f \geq f_{0} \}$ is diffeomorphic to $[f_{0},+\infty)\times \{f=f_{0}\}$. Hence, the metric $g$ can be written as 
$$
g \,=\, \frac{df \otimes df}{|\nabla f|^{2}} + h_{\alpha\beta}(f,\theta) \,d \theta^{\alpha} \otimes d \theta^{\beta}\,,
$$
where $\theta=(\theta^{1},\ldots,\theta^{n-1})$ are local coordinates on the regular level set $\{f=f_{0}\}$ and $h (f,\cdot)_{|p}$ is the metric induced on the regular level set $\{f=f(p)\}$. Let $u$ be a smooth function on $M$. Then, it is well known that the Laplacian of $u$ can be written as
\begin{equation}
\label{lapsigma}
\Delta u \,=\, \nabla^{2} u \,( \mathbf{n}, \mathbf{n} ) + \HHH \, \langle \nabla u, \mathbf{n} \rangle + \Delta^{h} u \,,
\end{equation}
where $\HHH(p)$ is the mean curvature of the level set $\{f=f(p)\}$. We notice that $ \nabla^{2} |T|^{2} \,( \mathbf{n}, \mathbf{n} ) = O\, (f^{-1})$ and that
$$
 H \, =\, \frac{\Delta f - \nabla^{2} f (\mathbf{n}, \mathbf{n} )}{|\nabla f|} \,=\, \frac{n-1-\RRR+\Ric(\mathbf{n},\mathbf{n})}{2|\nabla f|} \,=\, O\,(f^{-1/2})\,.  
$$
To proceed, we notice that from Corollary~\ref{cor1} one has $\langle \nabla |T|^{2}, \nabla f \rangle \leq C_{1}$, for some $C_{1}>0$. In particular, from identity~\eqref{equ:2}, $\partial_{f} |T|^{2} \leq C_{2} f^{-1}$ for some $C_{2}>0$.
Morever, from Proposition~\ref{PDE}, we have
\begin{eqnarray*}
f\, \partial_{f} |T|^{2} &=& \langle \nabla |T|^{2}, \nabla f \rangle+ R \partial_{f} |T|^{2} \\
&\leq& \Delta |T|^{2} - c(n) |T|^{2} + C_{3} f^{-1} \\
&\leq& \Delta^{h} |T|^{2} - c(n) |T|^{2} + C_{4} f^{-1} \,,
\end{eqnarray*}
for some $C_{3},C_{4}>0$. Setting $s:=\log f$, $s_{0}= \log f_{0}$ and $v(s,\theta):=|T|^{2}(e^{s},\theta)$, we have
\begin{equation}
\label{PDEv}
\partial_{s} v \, \leq \, \Delta^{h} v - c(n)\, v + C_{4} \,e^{-s} \,.
\end{equation}
Since $\{f=f_{0}\}$ is compact, it is well defined $v_{0}:=\max_{\theta} v (s_{0},\theta) $. By the parabolic maximum principle one has that $0\leq v(s,\theta) \leq V(s)$, where $V(s)$ is the unique solution of the ODE associated to~\eqref{PDEv} with initial condition $V(0)=v_{0}$, namely 
$$
V(s):= \frac{C_{4}}{c(n)-1}e^{-s}+\Big( v_{0}\,e^{c(n)\,s_{0}}- \frac{C_{4}}{c(n)-1}e^{(c(n)-1)s_{0}}\Big)\,e^{-c(n)s
} \,.
$$
Recalling that $c(n)= 2/(n-2) - \eta$, we set $a(n):=\min\{1, 2/(n-2) - \eta \}>0$ and we obtain $0\leq v(s,\theta) = O\,(e^{-a(n)s})$ as $s\to \infty$. Rephrasing this in terms of $|T|$ and $f$, we have proved that $|T|^{2} = O\, (f^{-a(n)})$. Combining the standard interpolation inequalities as in~\cite[Corollary 12.6]{Ham-3d-Pos} with Sobolev estimates and using the fact that the metric is asymptotically cylindrical it is possible to prove the $C^{k}$-estimate. For instance, if $k=1$, we consider compact domains of the form $\Omega_{r,s}=\{p\in M^{n}|\,\dist(p,\{f=r\})\leq s \}$ and a cutoff functions $\chi\in C^{\infty}_{c}(\Omega_{r,2})$ with $\chi\equiv 1$ on $\Omega_{r,1}$. By interpolation inequalities, we have
$$
|| \nabla^{2} (\chi\, T) ||_{L^{p}(\Omega_{r,2})} \leq C(n,p) \, ||\chi\,T||^{1-2/p}_{L^{\infty}(\Omega_{r,2})} \, ||\nabla^{p} (\chi \,T)||^{2/p}_{L^{2}(\Omega_{r,2})} \leq C'(n,p) ||T||^{1-2/p}_{L^{\infty}(\Omega_{r,2})} \,,
$$
where in the last inequality we have used the fact that the soliton is asymptotically cylindrical and hence the derivatives of the curvature are bounded. For $p$ large enough, Sobolev inequality implies that there exist a constant $C_{S}(\Omega_{r,2})$ such that
$$
||\nabla (\chi\,T)||_{L^{\infty}(\Omega_{r,2})} \leq C_{S}(\Omega_{r,2}) \, || \nabla^{2} (\chi\, T) ||_{L^{p}(\Omega_{r,2})} \,.
$$
Again, it is not hard to see that $C_{S}:=\sup_{r>0}C_{S}(\Omega_{r,2}) <+\infty$, since the soliton is asymptotically cylindrical. In particular, this shows that
$$
||\nabla (\chi\,T)||_{L^{\infty}(\Omega_{r,2})} \leq C_{S}\,C'(n,p) ||T||^{1-2/p}_{L^{\infty}(\Omega_{r,2})} \leq C'' r^{-a(n)(1-2/p)}\,,
$$
for some positive constant $C''$. This concludes the proof of the proposition.

\end{proof}

In dimension greater than three, one can estimate the scalar curvature from $T$. Notice that in dimension three, gradient shrinking Ricci solitons have been completely classified (see ~\cite{caochenzhu}).
\begin{cor}\label{Rdecay}
Let $(M^n, g, f)$, $n\geq 4$, be a complete noncompact asymptotically cylindrical gradient shrinking Ricci soliton. Then, we have that
\begin{eqnarray*}
|\nabla^{k} \RRR|=\textit{O}(f^{-a(n)/2+\eps})\quad\quad \hbox{and} \quad \quad |\nabla^{k}\Ric|=\textit{O}(f^{-a(n)/2+\eps})
\end{eqnarray*}
for every $k>0$ and every $\eps>0$.
\end{cor}

\begin{proof} Reasoning as the previous proposition, it is not hard to see that
\begin{eqnarray*}
\nabla_{i} T_{ij} \,=\, \frac{n-3}{2(n-1)}\nabla_{j} \RRR + \frac{1}{n-1}\nabla_{i} \big(\RRR\, \mathbf{n}_{i}\, \mathbf{n}_{j} \big) \,=\,  \frac{n-3}{2(n-1)}\nabla_{j} \RRR + \textit{O}(f^{-1/2}) \,.
\end{eqnarray*}
By Proposition~\ref{Tdecay}, there exists a positive constant $a(n)$ such that $$|\div T| \leq |\nabla T| = O (f^{-a(n)/2+\eps}),$$ for all $\eps>0$. And this proves the first estimate. To obtain the estimate for $|\nabla \Ric|$ it is sufficient to compute $|\nabla T|$ and use both the estimate for $|\nabla \RRR|$ and Proposition~\ref{Tdecay}. This proves the case $k=1$. For $k>1$ it is sufficient to repeat the same argument, based on interpolation inequalities and uniform Sobolev estimates, as in the proof of Proposition~\ref{Tdecay}.
\end{proof}

We are now in the position to improve the decay estimate of the scalar curvature at infinity.
\begin{prop}\label{sdecay} Let $(M^n, g, f)$, $n\geq 4$, be a complete noncompact asymptotically cylindrical gradient shrinking Ricci soliton. Then, we have that
\begin{eqnarray*}
\RRR=\frac{n-1}{2}+\textit{O}(f^{-a(n)/2+\epsilon}).
\end{eqnarray*}
\end{prop}

\begin{proof}
Recall that $\Delta \RRR+2\arrowvert \Ric\arrowvert^2=\langle \nabla f, \nabla \RRR \rangle+ \RRR$. Therefore, if $U:=(n-1)/2-\RRR$, then $U$ satisfies
\begin{eqnarray*}
\langle \nabla f, \nabla U\rangle &=& -\Delta \RRR+\RRR-2\left(\arrowvert T\arrowvert^2+\frac{\RRR^2}{n-1}-\frac{2\RRR}{n-1}T(\mathbf{n},\mathbf{n})\right)\\
&=&\frac{2\RRR}{n-1}U+\textit{O}(f^{-a(n)/2+\epsilon}),
\end{eqnarray*}
where we have used Proposition~\ref{T_decay},~\ref{Tdecay} and Corollary~\ref{Rdecay}.
Integrating this equality along the flow generated by $\nabla f/|\nabla f|^2$ and using the fact that $U$ tends to zero at infinity gives the result.
\end{proof}

\begin{prop}\label{prop-est-cur-op}
Let $(M^n, g, f)$, $n\geq 4$, be a complete noncompact asymptotically cylindrical gradient shrinking Ricci soliton. Then, we have that
\begin{eqnarray*}
\Rm=\frac{1}{2(n-2)}(g-\mathbf{n}\otimes\mathbf{n})\odot(g-\mathbf{n}\otimes\mathbf{n})+\textit{O}(f^{-a(n)/2+\epsilon}).
\end{eqnarray*}
\end{prop}

%
%

\begin{proof}
In the case of a shrinking soliton, Lemma \ref{lemm-sol-id} tells us
\begin{eqnarray*}
\nabla_{\nabla f}\Rm=\nabla^2\Ric-\Rm+\Rm\ast\Ric.
\end{eqnarray*}
Now, by the very definition of $T$, we get
\begin{eqnarray*}
\nabla_{\nabla f}\Rm=\nabla^2\Ric+\Big(\frac{2\RRR}{n-1}-1\Big)\Rm+\Rm\ast T.
\end{eqnarray*}
Therefore, by Proposition~\ref{T_decay},~\ref{Tdecay} and Corollary~\ref{Rdecay}, one has
\begin{eqnarray*}
\nabla_{\nabla f}\Rm=\textit{O}(f^{-a(n)/2+\epsilon}).
\end{eqnarray*}
As 
$$
\nabla_{\nabla f}[(g-\mathbf{n}\otimes\mathbf{n})\odot(g-\mathbf{n}\otimes\mathbf{n})]=\textit{O}(\Ric(\mathbf{n},\cdot))=\textit{O}(f^{-1/2})\,,
$$
one has
\begin{eqnarray*}
\nabla_{\nabla f}(\Rm-\frac{1}{2(n-2)}(g-\mathbf{n}\otimes\mathbf{n})\odot(g-\mathbf{n}\otimes\mathbf{n}))=\textit{O}(f^{-a(n)/2+\epsilon}).
\end{eqnarray*}
Integrating this estimate along the flow generated by $\nabla f/|\nabla f|^2$ we conclude the proof of the proposition.
\end{proof}

Let $M_{t}$ be a connected component of the level set $\{f=t\}$ and let $g^{(t)}$ be the metric induced by $g$ on $M_{t}$. We notice that the second fundamental form of $M_{t}$ satisfies
$$
h^{(t)}_{ij}=\frac{\nabla_{i}\nabla_{j} f}{|\nabla f|} = \textit{O}(t^{-1/2})\,.
$$
Thus, combining the last proposition with Gauss equations, we obtain
\begin{eqnarray*}
\RRR^{(t)}_{ijkl} &=& \RRR_{ijkl}-h^{(t)}_{jk}h^{(t)}_{il}+h^{(t)}_{ik}h^{(t)}_{jl} = \frac{1}{2(n-2)}(g^{(t)}\odot g^{(t)})_{ijkl} + \textit{O}(t^{-a(n)/2+\eps}) \\
&=& \frac{\RRR^{(t)}}{(n-1)(n-2)}(g^{(t)}\odot g^{(t)})_{ijkl} +\textit{O}(t^{-a(n)/2+\eps})\,,
\end{eqnarray*}
where we have used Proposition~\ref{sdecay} in the last equality. As a consequence of the Riemann-Cartan uniformization theorem, for $t$ large enough, there exists a family of diffeomorphisms $\phi_{t}:  M_{t}\rightarrow \mathbb{S}^{n-1}$ such that
$$
||2(n-2)g^{(t)}-\phi^{*}_{t}g^{\mathbb{S}^{n-1}}||_{C^{k}(M_{t}, g^{(t)})} = \textit{O}(t^{-a(n)/2+\eps})\,,
$$
for every $k\geq 0$ and every $\eps>0$.

\section{Almost Killing fields at infinity}\label{Alm-Kill-Fiel}

It well known that on the $(n-1)$-dimensional sphere there are $n(n-1)/2$ linearly independent Killing vector fields, hence, by construction the same is true on $(M_{t},\phi^{*}_{t}g^{\mathbb{S}^{n-1}})$, for $t$ large enough. The aim of the following sections is to show that for some $t_0$, the $(n-1)$-dimensional manifold $(M_{t_0},g^{(t_0)})$ admits $n(n-1)/2$ Killing vector fields as well. By classical results, this will imply that $(M_{t},g^{(t)})$ must be homothetic to the round sphere $\mathbb{S}^{n-1}$.
%

We consider now the sequence $t_m = 2^m$ and the corresponding sequence of Riemannian manifolds $(M_{t_m^2/4}, \tilde g ^{m})$, where we set 
$$
\tilde{g}^{m}:=\tfrac{1}{2(n-2)}\phi^{*}_{t_m^2/4}g^{\mathbb{S}^{n-1}} \, .
$$
We then let $\{ U_i^m\}_{i= 1, \ldots , n(n-1)/2}$ be a collection of linearly independent Killing vector fields for $\tilde g^m$. By the results of the previous section, these vector fields can be regarded as approximate Killing vector fields on $M_{t_m^2/4}$ for the matrix $g^{t^2_m/4}$. In particular, it is possible to prove that 
\begin{itemize}
\item $|| U^m_i ||_{C^k ( \,M_{t_m^2/4} \, , \,  g^{(t^2_m/4)} \,)} \, = \, \textit{O}\, (1)$, for every $k\geq 0$.
\item $\Li_{U_i^m} \, g^{(t^2_m/4)} \, = \, \textit{O} \, (t_m^{- a(n) +\epsilon})$\,. \phantom{$|| U^m_i ||_{C^k ( \,M_{t_m^2/4} \, , \,  g^{(t^2_m/4)} \,)} \, = \, \textit{O}(1)$}
\item $\int_{M_{t_m^2/4}}\langle U_i^m,U_j^m\rangle \, d\tilde{\mu}_{m} \, = \, \vol(M_{t_m^2/4}, \tilde g^m) \, \delta_{ij} $.
\end{itemize}

In the next proposition, we are going to extend these estimates to an annulus bounded by the level sets $M_{t^2_m/4}$. To do that, we define $\Omega_m:= \{ t^2_m/4 \leq f \leq t^2_m\}$, so that $\partial \Omega_m = M_{t_m^2/4} \cup M_{t_m^2}$ and we extend the vector fields $\{U_i^m\}$ on $\Omega_m$ by imposing the condition $[U_i^m,X]=0$. With a small abuse of notation, we still denote by $\{U^m_i\}$ the extended vector fields.
\begin{prop}
Let $(M^n, g, f)$, $n\geq 3$, be a complete noncompact asymptotically cylindrical gradient shrinking Ricci soliton. Then, with the notation introduced above, we have that for every $m \in \mathbb{N}$ and every $i \in \{1, \ldots, n(n-1)/2\}$ the following estimates hold
\begin{eqnarray*}
&&(i) \,\, \quad || U^m_i ||_{C^k (\Omega_m \, , \, g \, )} \, = \, \textit{O} \, (1) \, ,  \quad \hbox{for every $k\geq0$}, \\
&&(ii)  \, \quad \hbox{$\sup_{\Omega_m}$} | \Li_{U_i^m} g \,  | \, = \, \textit{O} \, (t_m^{-a(n)+\epsilon}),\\
&&(iii) \quad \hbox{$\sup_{\Omega_m}$} \arrowvert \langle U_i^m,\mathbf{n}\rangle\arrowvert \, = \, \textit{O} \, (t_m^{- 1 - a(n) +\epsilon}),\\
&&(iv) \, \quad\Delta_f U^m_i+U^m_i/2 \, = \, \textit{O} \,(t_m^{-a(n)+\epsilon}) ,  \\
&&(v) \, \quad\int_{M_t} \langle U_i^m,U_j^m\rangle d\mu_{t} \, =  \, \vol(M_t , g ^{(t)}) \, \delta_{ij}+\textit{O} \, (t^{-a(n)/2+\epsilon}), \quad\hbox{for every $t\in[t_m^2/4,t_m^2]$} \,.
\end{eqnarray*}
\end{prop}

\begin{rem}
\label{remeq}
We notice that in general, a Killing field $U$ on a Riemannian manifold satisfies $\Delta U+\Ric(U, \cdot\,)=0$.
Using the shrinking solitons equation and recalling that, by construction we have $[U_i^m,X]=0$, one can deduce that $\Delta U_i^m+\Ric(U_i^m ,\, \cdot \, )=\Delta_fU_i^m+U_i^m/2$. Therefore, part $(iv)$ of the statement can be thought as an estimate of how far the vector fields $U_i^m$ are from being Killing. 
\end{rem}

\begin{proof} $(i)$ This statement is a consequence of the the equations $[U_i^m , X ]=0$, which we have used to extend our vector fields. 

\smallskip

$(ii)$ Let us check this estimate on $M_{t_m^2/4}$ first. Indeed, if $V$ is orthogonal to $X$,
\begin{eqnarray*}
\left( \Li_{U_i^m}g \right) \, (X,V)&=&\langle \nabla_XU_i^m,V\rangle+\langle \nabla_VU_i^m,X\rangle\\
&=&\langle\nabla_{U_i^m}X,V\rangle-\langle U_i^m,\nabla_VX\rangle=0.
\end{eqnarray*}
To proceed, we compute
\begin{eqnarray*}
\left(\Li_{U_i^m}g \right)(X,X)=2\langle\nabla_XU_i^m,X\rangle=2\nabla^2f(U_i^m,X)=-2\Ric(U_i^m,X).
\end{eqnarray*}
Therefore, $ \left( \Li_{U_i^m}g \right)(\mathbf{n},\mathbf{n})=\textit{O}(t_m^{-2 - a(n)+ \epsilon})$ on $M_{t_m^2/4}$.
Now, as $[U_i^m,X]=0$,
\begin{eqnarray*}
\Li_X(\Li_{U_i^m} g)=\Li_{U_i^m}(\Li_X g)=\Li_{U_i^m}(g-2\Ric).
\end{eqnarray*}
Moreover, $(\Li_X S) \, = \, \nabla_X S +  S -\Ric \, \circ \,  S - S\circ\Ric$, for any symmetric $2$-tensor $S$. Therefore, for $S:=\Li_{U_i^m}g$, we get
$$
\nabla_X(\Li_{U_i^m}g ) = -2 (\Li_{U_i^m}\Ric ) + \Ric \circ  (\Li_{U_i^m} g ) + (\Li_{U_i^m}g ) \circ \Ric  .
$$ 
Consequently, as $( \Li_{U_i^m} \Ric)$ is bounded, one has by Kato inquality, 
\begin{eqnarray*}
\nabla_{X/|X|^2} \left| \left(\Li_{U_i^m}g \right) \right| \leq 
\left|  \nabla_{X/|X|^2} \left(\Li_{U_i^m}g \right) \right|
\leq \frac{2\arrowvert\Ric\arrowvert}{\arrowvert X\arrowvert^2} \,  \left| (\Li_{U_i^m}g ) \right| +\textit{O}(f^{-1}).
\end{eqnarray*}
Hence, we obtain the result by integrating over $\Omega_m$.

\smallskip

$(iii)$ As before, we compute,
\begin{eqnarray*}
\nabla_X \langle U_i^m,X\rangle&=&\langle \nabla_XU_i^m,X\rangle+\langle U_i^m,\nabla_XX\rangle\\
&=&2\langle \nabla_{U_i^m}X,X\rangle=\langle U_i^m,X\rangle-2\Ric(X,U_i^m)\\
&=&\langle U_i^m,X\rangle+\textit{O}(f^{-a(n)/2+\epsilon}).
\end{eqnarray*}
Now, by construction, we have that $\langle U_i^m,X\rangle=0$ on $M_{t_m^2/4}$. Hence, integrating the previous estimate on $\Omega_m$, we obtain $\langle U_i^m,X\rangle \,= \,\textit{O}(t_m^{-a(n)+\epsilon})$ on $\Omega_m$ and the result follows at once.

\smallskip

$(iv)$ As we noticed in Remark~\ref{remeq}, one has that $\Delta U_i^m+\Ric(U_i^m , \cdot ) \, = \, \Delta_fU_i^m+U_i^m/2$. On the other hand, it holds the identity  
$$
\div(\Li_{U_i^m}g ) \, - \, \frac12 \nabla(\tr(\Li_{U_i^m}g )) \,\, = \,\, \Delta U_i^m \, + \,\Ric(U_i^m , \cdot ) \, .
$$
To estimate the left hand side, we notice that (ii) gives $\hbox{$\sup_{\Omega_m}$} | \Li_{U_i^m} g \, | \, = \, \textit{O} \, (t_m^{-a(n)+\epsilon})$. Moreover, it is possible to deduce form (i) that $ \sup_{\Omega_m}\arrowvert\nabla^{k}\Li_{U_i^m}(g)\arrowvert =\textit{O}\, (1)$, for every $k\geq 0$. By the interpolation inequalities, we know that $ \sup_{\Omega_m}\arrowvert\nabla\Li_{U_i^m}(g)\arrowvert=\textit{O}(t_m^{-a(n)+\epsilon})$. This implies the desired estimates.

\smallskip

$(v)$ To see the last estimate, we denote by $\HHH^{(t)}$ the mean curvature of $M_t$ and we compute
\begin{eqnarray*}
\frac{d}{dt}\int_{M_t}\langle U_i^m,U_j^m\rangle \, d\mu_t & = & \int_{M_t}\langle U_i^m,U_j^m\rangle \, \frac{\HHH^{(t)}}{\arrowvert X\arrowvert} \, d\mu_t +   \int_{M_t}\frac{\langle\nabla_XU_i^m,U_j^m\rangle+\langle U_i^m,\nabla_XU_j^m\rangle}{\arrowvert X\arrowvert^2}   \, d\mu_t \\
&=&\int_{M_t}\langle U_i^m,U_j^m\rangle\frac{(n-1)/2- \RRR+\Ric(\mathbf{n},\mathbf{n})}{\arrowvert X\arrowvert^2}d\mu_t +  \int_{M_t}\frac{ 2 \, \langle\nabla_{U_i^m}X,U_j^m\rangle}{\arrowvert X\arrowvert^2} \, d\mu_t
\\
&=&\textit{O}\, (t^{-1 - a(n)/2 +\epsilon}) + \int_{M_t}\frac { \langle U_i^m,U_j^m\rangle  -  2 \Ric(U_i^m,U_j^m)}{\arrowvert X\arrowvert^2} \,  d\mu_t\\
&=& \textit{O}\, (t^{-1 - a(n)/2 +\epsilon}) + \int_{M_t}\frac{ \langle U_i^m,\mathbf{n}\rangle\langle U_j^m,\mathbf{n}\rangle - 2 \, T(U_i^m,U_j^m)}{\arrowvert X\arrowvert^2} \, d\mu_t\\
&=& \textit{O}\, (t^{-1 - a(n)/2 +\epsilon}) ,
\end{eqnarray*}
where we used the estimates obtained in the previous section. The result follows now by a simple integration.
\end{proof}
For future convenience, we simplify the notations and summarize the results of this section in the following proposition.

\begin{prop}
\label{prop-ex-alm-kill}
Let $(M^n, g, f)$, $n\geq 3$, be a complete noncompact asymptotically cylindrical gradient shrinking Ricci soliton. Then, there exists a collection of $n(n-1)/2$ vector fields $\{U_i \}_{i=1, \ldots, n(n-1)/2}$ defined on $(M^n,g)$ such that, for every $i \in \{1, \ldots, n(n-1)/2\}$, the following estimates hold
\begin{eqnarray*}
(i)\quad && |\nabla^kU_i \,| \, = \, \textit{O} \, (1), \quad \mbox{for every $k\geq 0$,}\\
(ii) \quad&& |\Li_{U_i}g \, | \, = \, \textit{O} \, (f^{-a(n)/2+\epsilon}),\\
(iii)\quad &&\Delta_f U_i+U_i/2 \, = \, \textit{O} \,(f^{-a(n)/2+\epsilon}),\\
(iv)\quad &&\int_{M_t}\langle U_i,U_j\rangle \,d\mu_t \, = \, \vol(M_t, g^{(t)})	\, \delta_{ij}+\textit{O}\, (t^{-a(n)/2+\epsilon}).
\end{eqnarray*}
\end{prop}

\section{Interpolating almost Killing vector fields}\label{Vec-Interpol}

In this section, we will provide an a priori estimate as well as an existence result for solutions to the following equation
\begin{eqnarray}
\label{akvf}
\Delta_fV+{V}/{2} \, = \, Q,
\end{eqnarray}
provided the vector field $Q$ has a suitable asymptotic behavior at infinity. As we have already seen, the almost Killing vector fields constructed in the previous section satisfy an equation of this type. We start with the a priori estimate.

\begin{prop}\label{max-ppe-vec} Let $(M^n, g, f)$, $n\geq 3$, be a complete noncompact asymptotically cylindrical gradient shrinking Ricci soliton and let $V$ be a vector field satisfying
\begin{eqnarray*}
\Delta_fV+ {V}/{2} \, = \, Q,
\end{eqnarray*}
where $Q$ is vector field such that $Q=\textit{O}(f^{-a(n)/2+\epsilon})$. Then, there exists positive constants $A$ and $t_0$ such that, for $t_0<t_1<t_2$, one has the following estimate
\begin{eqnarray*}
\max_{t_1\leq f\leq t_2}\frac{\arrowvert V\arrowvert+Af^{-a(n)/2+\epsilon}}{f-n/2}=\max\left\{\max_{f=t_1}\frac{\arrowvert V\arrowvert+Af^{-a(n)/2+\epsilon}}{f-n/2};\max_{f=t_2}\frac{\arrowvert V\arrowvert+Af^{-a(n)/2+\epsilon}}{f-n/2}\right\}.
\end{eqnarray*}

\end{prop}

\begin{proof}
First of all, 
\begin{eqnarray*}
\Delta_f\arrowvert V\arrowvert^2&=&-\arrowvert V\arrowvert^2+2\arrowvert\nabla V\arrowvert^2+2\langle Q,V\rangle.
\end{eqnarray*}
Therefore,
\begin{eqnarray*}
2\arrowvert V\arrowvert\Delta_f\arrowvert V\arrowvert\geq- \arrowvert V\arrowvert^2+ 2\arrowvert\nabla V\arrowvert^2-2\arrowvert \nabla\arrowvert V\arrowvert\arrowvert^2-2\arrowvert Q\arrowvert \arrowvert V\arrowvert.
\end{eqnarray*}
Hence, by the Kato inequality,
\begin{eqnarray*}
\Delta_f\arrowvert V\arrowvert\geq- \frac{\arrowvert V\arrowvert}{2}-\arrowvert Q\arrowvert,
\end{eqnarray*}
as soon as $V$ does not vanish. 
\begin{eqnarray*}
\Delta_f(\arrowvert V\arrowvert+Af^{-\alpha})+\frac{\arrowvert V\arrowvert+Af^{-\alpha}}{2}&\geq& A\Delta_ff^{-\alpha}+\frac{A}{2}f^{-\alpha}-\arrowvert Q\arrowvert\\
&\geq&A\alpha f^{-\alpha-1}(f-n/2)\\
&&+A\alpha(\alpha+1)f^{-\alpha-2}\arrowvert X\arrowvert^2+\frac{A}{2}f^{-\alpha} -\arrowvert Q\arrowvert\\
&\geq&f^{-\alpha}\left(A\left(\alpha+1/2-\frac{\alpha n}{2f}\right)-f^{\alpha}\arrowvert Q\arrowvert\right) \,\,\, > \,\,\ 0,
\end{eqnarray*}
outside a compact set where $\alpha:=a(n)/2-\epsilon$. Finally, consider the function $v:=f-n/2$, which satisfies $\Delta_fv=-v$, and define $u:=\arrowvert V\arrowvert+Af^{-a(n)/2+\epsilon}$. For $v>0$, a direct computation gives
\begin{eqnarray*}
\Delta_f\left(\frac{u}{v}\right)&=&\left(\frac{\Delta_fu}{u}-\frac{\Delta_fv}{v}\right)\frac{u}{v}+2\langle\nabla v^{-1},\nabla u\rangle\\
&>&(-1/2+1)\frac{u}{v}+2\langle\nabla v^{-1},\nabla u\rangle\\
&>&\left(1/2-2\frac{\arrowvert X\arrowvert^2}{f^2}\right)\frac{u}{v}-2\langle\nabla\ln v,\nabla\left(\frac{u}{v}\right)\rangle\\
&>&-2\langle\nabla\ln v,\nabla\left(\frac{u}{v}\right)\rangle,
\end{eqnarray*}
outside a compact set. The result is now a consequence of the maximum principle.
\end{proof}

We are in the position to provide an existence result for the equation~\eqref{akvf}.

\begin{teo}\label{theo-ex-app-kill} Let $(M^n, g, f)$, $n\geq 3$, be a complete noncompact asymptotically cylindrical gradient shrinking Ricci soliton and let $Q$ be a vector field such that $Q=\textit{O}(f^{-a(n)/2+\epsilon})$. Then there exists a vector field $V$ such that,
\begin{eqnarray*}
\Delta_fV+ {V}/{2} \, = \, Q\quad\mbox{on $M^n$},  \qquad \hbox{and} \qquad V=\textit{O}\, (f^{-a(n)/2+\epsilon}).
\end{eqnarray*}
\end{teo}

\begin{proof}
Let $(t_m)_{m\in \mathbb{N}}$ be a sequence tending to $+\infty$ and, for every $m \in \mathbb{N}$, let $V^m$ be a solution of the Dirichlet problem
\begin{eqnarray*}
\Delta_f V^m+ {V^m}/{2} \, = \, Q\quad\hbox{in $\{\, f \, \leq \,  t_m^2/4\,  \}$}, \qquad  \hbox{with} \qquad  V^m =0 \quad \hbox{ on $\{ \, f = t_m^2/4 \, \}$}.
\end{eqnarray*}
To study the growth of $V^{m}$, we are going to prove the following claim.

\bigskip

\noindent{\bf Claim}.
{\em Given $\tau$ and $\alpha$ in $(0,1)$, there exists a positive constant $\rho_0$ such that, for every $t\in[\rho_0,t_m]$, it holds the estimate
\begin{eqnarray*}
\alpha^{-1}\tau^{a(n)-2\epsilon} \, \bar{\A}^m(\tau t) \,\, \leq \,\, \bar{\A}^m(t) \, + \, t^{-a(n)+2\epsilon},
\end{eqnarray*}
where we set $\bar{\A}^m(s):=\sup_{ \{ \, f \, = \, s^2/4 \, \}}\arrowvert V^m\arrowvert$.}

\bigskip

Assume by contradiction that there is a sequence $(r_m)_{m\in \mathbb{N}}$ going to $+\infty$ such that
\begin{eqnarray}\label{ineq-abs-vec}
\alpha^{-1}\tau^{a(n)-2\epsilon}\bar{\A}^m(\tau r_m)\geq\bar{\A}^m(r_m)+r_m^{-a(n)+2\epsilon}
\end{eqnarray}
an define
\begin{eqnarray*}
&&\bar{V}^m:=\frac{V^m}{\max\{\bar{\A}^m(\tau r_m)+(\tau r_m)^{-a(n)+2\epsilon},\tau^2(\bar{\A}^m( r_m)+r_m^{-a(n)+2\epsilon})\}},\\
&&\quad f_m:=\frac{f^{-a(n)/2+\epsilon}}{\max\{\bar{\A}^m(\tau r_m)+(\tau r_m)^{-a(n)+2\epsilon},\tau^2(\bar{\A}^m( r_m)+r_m^{-a(n)+2\epsilon})\}}.
\end{eqnarray*}
By Proposition~\ref{max-ppe-vec}, we get,
\begin{eqnarray*}
1\leq \sup_{(\tau t_m)^2/4\leq f\leq t_m^2/4}\arrowvert \bar{V}^m\arrowvert+Af_m\leq \tau^{-2}.
\end{eqnarray*}
Moreover, $\bar{V}^m$ satisfies
\begin{eqnarray*}
&&\Delta_f \bar{V}^m+\frac{\bar{V}^m}{2}=\bar{Q}^m,\\
&&\bar{Q}^m:=\frac{Q}{{\max\{\bar{\A}^m(\tau r_m)+(\tau r_m)^{-a(n)+2\epsilon},\tau^2(\bar{\A}^m( r_m)+r_m^{-a(n)+2\epsilon})\}}}.
\end{eqnarray*}
By assumption on the growth of $Q$, the sequence $\bar{Q}^m$ is uniformly bounded.
Therefore, as $(M^n,g)$ is asymptotically cylindrical, by blowing-up this equation we obtain that $(\bar{V}^m)_{m\in \mathbb{N}}$ converges to a vector field $\bar{V}^{\infty}$ which is radially constant, i.e. $\nabla_{\partial_r}\bar{V}^{\infty}=0$. Observe also that, if $f_{\infty}:=\lim_{m\rightarrow +\infty}f_m$,
\begin{eqnarray*}
\nabla_{\partial r} f_{\infty}=\lim_{m\rightarrow +\infty}\nabla_{\nabla{2\sqrt{f}}}f_m=0.
\end{eqnarray*}
Consequently, the supremum of $\arrowvert \bar{V}^{\infty}\arrowvert$ on each slice of the cylinder is a nonnegative constant $c_{\infty}$ independent of the slice, the same holds for $f_{\infty}$. Now, inequality~\eqref{ineq-abs-vec} reads, as $m$ tends to $+\infty$,
\begin{eqnarray*}
\alpha^{-1}\tau^{a(n)-2\epsilon}c_{\infty}\geq c_{\infty}+Af_{\infty}>0.
\end{eqnarray*}
In particular, $c_{\infty}>0$, i.e. $\bar{V}^{\infty}$ does not vanish identically and
\begin{eqnarray*}
\alpha^{-1}\tau^{a(n)-2\epsilon}\geq 1,
\end{eqnarray*}
which is a contradiction if $\alpha$ and $\tau$ are chosen properly. This proves the claim.

\bigskip

The claim ensures that $\sup_{m}\sup_{t\in[\rho_0,t_m]}\sup_{ \{ \, f \, = \, t^2/4 \, \}}f^{a(n)/2-\epsilon} \, \arrowvert V^m\arrowvert<+\infty$.
Therefore, if the sequence $(V^m)_{m \in \mathbb{N}}$ is not uniformly bounded on $M$, this can only happen on a fixed compact set. Assume on the contrary that $\sup_m\sup_{M^n}\arrowvert V^m\arrowvert=+\infty$. Define $W^m:=V^m/\sup_{M^n}\arrowvert V^m\arrowvert$. Since $Q$ is bounded on $M^n$, this implies that $(W^m)_{m\in \mathbb{N}}$ uniformly converges on compact sets to a non vanishing vector field $W^{\infty}$ with compact support satisfying $\Delta_fW^{\infty}+W^{\infty}/2=0$. Now, by the work of Bando \cite{Ban-Ana}, $\Delta_f$ is an elliptic operator with analytic coefficients, therefore, $W^{\infty}$ must be analytic too. Since it has compact support, it must vanishes everywhere which is a contradiction.
\end{proof}

\section{Rigidity of the Lichnerowicz equation}\label{Rig-Lic-Equ}


We start this section with the following proposition, which provides an elliptic equation for the Lie derivative of the metric along a vector field which satisfy equation~\eqref{akvf} with $Q=0$.
\begin{prop} Let $(M^n, g, f)$, $n\geq 3$, be a  gradient shrinking Ricci soliton. Assume a vector field $V$ satisfies
\begin{eqnarray}\label{eq-quasi-kill} 
\Delta_fV+ {V}/{2} \, = \, 0.
\end{eqnarray}
Then the Lie derivative $h:=(\Li_Vg)$ satisfies the Lichnerowicz equation
\begin{eqnarray}
\label{liceq}
(\Li_X h)-h=\Delta_Lh,
\end{eqnarray}
where $\Delta_Lh$ denotes the Lichnerowicz laplacian.
\end{prop}

\begin{proof}
Consider the flow $\{\phi_t\}_t$ generated by the vector field $V$ and the family of metrics $g(t):=\phi_t^{\ast}g$. By equation $(2.31)$ in~\cite[Chapter 2]{Ben} we obtain the variation of the Ricci curvature at the initial time.
\begin{eqnarray*}
\frac{\partial}{\partial t}{\Big|_{t=0}} (-2\Ric_{g(t)})&=&\Delta_L h +\Hess \tr(h)-\Li_{\div(h)}(g)\\
&=&\Delta_L h -\Li_{\div(h)-\frac{1}{2}\nabla\tr(h)}(g),
\end{eqnarray*}
where we set $h:=\frac{\partial}{\partial t}g(t)\arrowvert_{t=0}$. Since $h= ( \Li_Vg)$, we have that
\begin{eqnarray}\label{eq-div-li}
\div(h)-\frac{1}{2}\nabla\tr(h)=\Delta V+\Ric(V).
\end{eqnarray}
Using both the equation~(\ref{eq-quasi-kill}) satisfied by the vector field $V$ and the soliton equation, we deduce
\begin{eqnarray*}
\div(h)-\frac{1}{2}\nabla\tr(h) \,\, = \,\, \Delta V-\nabla_VX+ {V}/{2}  \,\,= \,\, [X,V].
\end{eqnarray*}
We can conclude that
\begin{eqnarray*}
-2( \Li_V\Ric) \, = \, \Delta_L(\Li_Vg) - (\Li_{[X,V]}g).
\end{eqnarray*}
Recalling that $2\Ric = - (\Li_Xg)+g$, one has
$-\Li_V(-(\Li_Xg)+g)=\Delta_L(\Li_Vg )- (\Li_{[X,V]}g)$
an thus
\begin{eqnarray*}
-(\Li_Vg)+\Li_X(\Li_V(g))=\Delta_L(\Li_V(g)),
\end{eqnarray*}
which is the desired equation.
\end{proof}

We conclude this section with the analysis of the Lichnerowicz equation~\eqref{liceq}, providing an a priori estimate and a Liouville-type theorem.

\begin{prop}\label{c0-est-lic} Let $(M^n, g, f)$, $n\geq 3$, be a complete noncompact asymptotically cylindrical gradient shrinking Ricci soliton and let $h$ be a symmetric $2$-tensor satisfying the static Lichnerowicz equation
\begin{eqnarray*}
\Li_X(h)-h=\Delta_Lh . 
\end{eqnarray*}
Then there exists $\alpha>0$ large enough such that for $t_0\leq t_1<t_2$,
\begin{eqnarray*}
\max_{t_1\leq f\leq t_2}\frac{\arrowvert h\arrowvert^2}{v^{\alpha}} \,\, = \,\, \max\, \left\{ \, \max_{f=t_1}\frac{\arrowvert h\arrowvert^2}{v^{\alpha}} \, ; \, \max_{f=t_2}\frac{\arrowvert h\arrowvert^2}{v^{\alpha}} \, \right\}.
\end{eqnarray*}
\end{prop}

\begin{proof} We start by observing that on a shrinking soliton, the Lichnerowicz equation can be rewritten as
$$
\Delta_fh+2\Rm\ast h \,\, = \,\, 0.
$$
Define as before $v:=f-n/2$. Then, we have the following estimate
\begin{eqnarray*}
\Delta_f\left(\frac{\arrowvert h\arrowvert^2}{v^{\alpha}}\right)&=&v^{-\alpha}\left\langle 2\Delta_fh-\left(\alpha\frac{\Delta_fv}{v}-\alpha(\alpha+1)\arrowvert\nabla\ln v\arrowvert^2\right)h,h\right\rangle\\
&&+2v^{-\alpha}\arrowvert\nabla h\arrowvert^2+2\langle \nabla v^{-\alpha},\nabla\arrowvert h\arrowvert^2\rangle\\
&\geq&v^{-\alpha}\langle\left(\alpha +\alpha(\alpha+1)\arrowvert\nabla\ln v\arrowvert^2\right)h-4\Rm\ast h,h\rangle\\
&&-2\alpha\left\langle\nabla\ln v,\nabla\left(\frac{\arrowvert h\arrowvert^2}{v^{\alpha}}\right)\right\rangle-2\alpha^2\arrowvert\nabla\ln v\arrowvert^2\frac{\arrowvert h\arrowvert^2}{v^{\alpha}}\\
&\geq&\left(\alpha\left(1-(\alpha-1)\arrowvert\nabla\ln v\arrowvert^2\right)-c(\arrowvert\Rm\arrowvert_{\infty})\right)\frac{\arrowvert h\arrowvert^2}{v^{\alpha}}-2\alpha\left\langle\nabla\ln v,\nabla\left(\frac{\arrowvert h\arrowvert^2}{v^{\alpha}}\right)\right\rangle\\
&\geq&-2\alpha\langle\nabla\ln v,\nabla\left(\frac{\arrowvert h\arrowvert^2}{v^{\alpha}}\right)\rangle,
\end{eqnarray*}
outside a compact set for $\alpha$ large enough. The result, follows then by the maximum principle.
\end{proof}

Building on the previous a priori estimates, we are now in the position to present a Liouville-type theorem for solutions to the Lichnerowicz equation with a suitable decay at infinity.

\begin{teo}\label{triv-ker} Let $(M^n, g, f)$, $n\geq 3$, be a complete noncompact asymptotically cylindrical gradient shrinking Ricci soliton and let $h$ be a symmetric $2$-tensor satisfying the static Lichnerowicz equation i.e.
\begin{eqnarray*}
\Li_X(h)-h=\Delta_Lh,
\end{eqnarray*}
such that $h=\textit{O}(f^{-\alpha_0})$, with $\alpha_0>0$. Then $h=0$.
\end{teo}

\begin{proof}
If there is a sequence $(t_m)_m$ tending to $+\infty$ such that $\sup_{M_{t_m^2/4}}\arrowvert h\arrowvert=0$ for any $m$, then $h=0$ on $f^{-1}([t_1,+\infty))$ by the a priori $C^0$ estimate given by Proposition \ref{c0-est-lic}. 

Assume by contradiction that $\A(t):=\sup_{M_{t^2/4}}\arrowvert h\arrowvert^2$ is positive for any large $t$. By the growth assumption on $h$, given $\beta\in(0,1)$ and $\tau\in(0,1)$, there is a sequence $(t_m)_m$ diverging to $+\infty$ such that
\begin{eqnarray}\label{abs-cond}
\beta^{-1}\tau^{2\alpha_0}\A(\tau t_m)\geq\A(t_m).
\end{eqnarray}

Now, define 
\begin{eqnarray*}
h_m:=\max\left\{\max_{f=(\tau t_m)^2/4}\arrowvert h\arrowvert;\tau^{\alpha}\max_{f=t_m^2/4}\arrowvert h\arrowvert\right\}.
\end{eqnarray*}
By Proposition \ref{c0-est-lic}, if $h^m:=h/h_m$,
\begin{eqnarray*}
1\leq\max_{(\tau t_m)^2/4\leq f\leq t_m^2/4}\arrowvert h^m\arrowvert^2\leq\frac{1}{\tau^{2\alpha}}.
\end{eqnarray*}
Moreover, $h^m$ still satisfies the Lichnerowicz equation. By blowing-up this equation, as $(M^n,g)$ is asymptotically cylindrical, $(h^m)_m$ converges to a symmetric $2$-tensor $h^{\infty}$ which is radially constant, i.e. $\nabla_{\partial_r}h^{\infty}=0$. In particular, the maximum of the norm of $h^{\infty}$ restricted to each slice of the cylinder is a positive constant denoted by $c_{\infty}$. Now, as $t_m$ goes to $+\infty$, inequality (\ref{abs-cond}) reads :
\begin{eqnarray*}
\beta^{-1}\tau^{2\alpha_0}c_{\infty}\geq c_{\infty},
\end{eqnarray*}
which is a contradiction if we choose $\beta$ and $\tau$ such that $\beta^{-1}\tau^{2\alpha_0}<1$. Therefore, $h$ vanishes outside a compact set and satisfies an elliptic equation with analytic coefficients, therefore, $h$ vanishes everywhere.
\end{proof}

\section{Conclusions}\label{Conc}
\subsection{Proof of the Theorem~\ref{main}}

With the notations and the results of Proposition~\ref{prop-ex-alm-kill}, one can apply Theorem~\ref{theo-ex-app-kill} to each vector field $U_i$ to ensure the existence of vector fields $V_i$ satisfying 
\begin{eqnarray*}
&&\Delta_fV_i+\frac{V_i}{2}=\Delta_f U_i+\frac{U_i}{2}\quad\mbox{on $M$},\\
&&V_i=\textit{O}(f^{-a(n)/2+\epsilon}).
\end{eqnarray*}

Therefore, $W_i:=U_i-V_i$ satisfies $\Delta_fW_i+W_i/2=0$. Moreover, the maximum principle applied to $V_i$ gives $\nabla V_i=\textit{O}(f^{-a(n)/2+\epsilon})$. Since, by construction, $\Li_{U_i}(g) =\textit{O}(f^{-a(n)/2+\epsilon})$, we get that $h_i:=\Li_{W_i}(g)=\textit{O}(f^{-a(n)/2+\epsilon})$. Consequently, Theorem~\ref{triv-ker} ensures that $h_i=0$ on $M$. Consequently, we have built $n(n-1)/2$ independent non trivial Killing vector fields on $M$. Morever, one has
\begin{eqnarray*}
[W_i,X]&=&\nabla_{W_i}X-\nabla_X W_i=\frac{W_i}{2}-\Ric(W_i)-\nabla_XW_i\\
&=&\frac{W_i}{2}+\Delta W_i-\nabla_XW_i=0.
\end{eqnarray*}
Now, if we let $\phi_{i}:=\langle W_i,X\rangle$, one has
\begin{eqnarray*}
2 \,\nabla^2 \phi_{i}&=&2\nabla^2\Li_{W_i}(f)=2\Li_{W_i}(\nabla^2 f)\\
&=&\Li_{W_i}(\Li_X(g))=\Li_{X}(\Li_{W_i}(g))=0.
\end{eqnarray*}
Then, by Kato inequality, we have that $|\nabla \phi_{i}|$ is constant on $M$. If, for some $1\leq i \leq n(n-1)/2$, $\nabla \phi_{i}$ is a nontrivial parallel vector field, then, it is easy to check that the metric $g$ is definitely isometric to a Riemannian product of a half line with a level set of $\phi_{i}$. By the fact that $g$ is asymptotically cylindrical, we conclude that the soliton is definitely isometric to the round cylinder and thus, by the analyticity, it is globally isometric to the round cylinder. On the other hand, if $|\nabla \phi_{i}|=0$ for every $i$, then $\phi_{i}=\langle W_i,X\rangle$ has to be constant. Thus,
\begin{eqnarray*}
0 & = & \langle\nabla f,\nabla \phi_{i}\rangle \,\,=\,\, \langle \nabla_{\nabla f} W_{i}, \nabla f \rangle + \nabla^{2} f (\nabla f, W_{i})  \\
&=& \frac{\phi_{i}}{2} - \Ric (\nabla f, W_{i}) \,,
\end{eqnarray*}
where we used the soliton equation and the fact that $W_{i}$ is Killing, and thus $\nabla W_{i}$ is antisymmetric. Hence, from equation~\eqref{equ:2} and Corollary~\ref{Rdecay},  we obtain
$$
\langle W_i,X\rangle \,=\, \phi_{i} \,=\, \langle W_i,\nabla \RRR\rangle \,=\, O(f^{-a(n)/2+\eps} ) \,.
$$
This implies $\langle W_i,X\rangle =0$ for every $i$. Hence, we have constructed $n(n-1)/2$ independent non trivial Killing vector fields on the level sets of $f$, for $f$ large enough. This implies that these level sets are isometric to $(n-1)$-dimensional round sphere, and by analyticity the soliton has to be globally isometric to the round cylinder. 

This completes the proof of Theorem~\ref{main}.

%
%
%

\subsection{Proof of Corollary \ref{nonneg-cyl}}

Let $(M^n, g, f)$ be a gradient shrinking Ricci soliton with bounded positive curvature operator. Assume $M^n$ is not compact. Then, by a result of Naber [Corollary 4.1, \cite{Nab-Sol}], we know that, for any sequence of points $\{x_k\}_{k\in \NN}$ tending to infinity, $(M^n,g,x_k)$ subconverges to $(\mathbb{R}\times N,dt^2+h,x_{\infty})$ where $(N,h)$ is a non flat gradient shrinking Ricci soliton with nonnegative curvature operator. If $(M^n,g)$ has linear volume growth, $N$ is compact and it is diffeomorphic to the levels of the potential function $f^{-1}(t)$ for large $t$. Now, since $M^n$ has positive curvature operator, $M^n$ is diffeomorphic to $\mathbb{R}^n$ by the Gromoll-Meyer theorem. In particular, $f^{-1}(t)$ is homeomorphic to a $(n-1)$-sphere. By the classification of compact manifolds with nonnegative curvature operator \cite{Boh-Wil}, we have that $N$ is diffeomorphic to the standard $(n-1)$-sphere and $h$ is the metric of positive constant curvature normalized by $\Ric_h=h/2$ and therefore, by Theorem \ref{main}, $(M^n,g)$ is isometric to the cylinder. In particular, it is not positively curved and we get a contradiction.

\

\bibliographystyle{amsplain}
\bibliography{bib-sgs-sym}

\end{document}